\documentclass[a4paper,draft,reqno,12pt]{amsart}
\usepackage[utf8]{inputenc}
\usepackage{amsmath,amsthm,amssymb}
\usepackage{mathtext}
\usepackage[T1,T2A]{fontenc}
\usepackage{tikz}

\usepackage{graphicx}
\usepackage{cite}
\usepackage{secdot}

\theoremstyle{plain}
\newtheorem{theorem}{Theorem}
\newtheorem{sled}{Corollary}
\newtheorem{lemma}{Lemma}
\newtheorem{propos}{Proposition}

\theoremstyle{definition}

\newtheorem{remark}{Remark}
\newtheorem{example}{Example}

\title[On the automorphism group of an affine toric variety]{On the connectedness of the automorphism group\\of an affine toric variety}
\author{Veronika Kikteva}
\address{HSE University, Faculty of Computer Science, 11 Pokrovsky Bulvar, Moscow, 109028, Russia}
\email{VVKikteva@yandex.ru}
\thanks{This research was supported by the Ministry of Science and Higher Education of the Russian Federation, agreement 075-15-2022-289 date 06/04/2022.}
\thanks{This is a preprint of the Work accepted for publication in Sbornik: Mathematics, 2024, Vol. 215, No 10. The owner of the distribution rights is Sbornik: Mathematics}

\begin{document}
\maketitle

\begin{abstract}
We obtain a criterion for the automorphism group of an affine toric variety to be connected in combinatorial terms and in terms of the divisor class group of the variety. The component group of the automorphism group of a non-degenerate affine toric variety is described. In particular, we show that the number of connected components of the automorphism group is finite.
\end{abstract}

\section{Introduction}
Let $\mathbb{K}$ be an algebraically closed field of characteristic zero and $X$ be an algebraic variety over the field $\mathbb{K}$. By $\mathrm{Aut}(X)$ we denote the group of regular automorphisms of the variety~$X$. In general, the group $\mathrm{Aut}(X)$ is not an algebraic group. However, for subgroups in the group $\mathrm{Aut}(X)$ the connectedness can be defined. This concept was introduced in~\cite{Ram}, see also~\cite{Popov}. Let $S$ be an irreducible affine algebraic variety. Then any map $S \to \mathrm{Aut}(X), s \mapsto \phi_s$ defines a \textit{family} $\{\phi_s\}_{s \in S}$ in the automorphism group of $X$, parameterized by the variety $S$. A family is called \textit{algebraic} if the map $S \times X \to X$, given by the rule $(s, x) \mapsto \phi_s(x)$, is a morphism of algebraic varieties. Let $G$ be a subgroup in $\mathrm{Aut}(X)$. If for every element $g \in G$ there exists an algebraic family $\{\phi_s\}_{s \in S}$ that contains $g$ and the identity automorphism, then $G$ is called a \textit{connected subgroup} in $\mathrm{Aut}(X)$. The \textit{neutral component} $\mathrm{Aut}(X)^0$ of the automorphism group of $X$ is the subgroup generated by elements of all algebraic families containing the identity automorphism. 

It follows from~\cite{IVSA10} that the automorphism group of a non-degenerate affine toric variety of dimension two or greater is infinite-dimensional and, therefore, it is not an algebraic group. In contrast to the affine case, the automorphism group of a complete toric variety is an affine algebraic group. Automorphism groups of complete simplicial toric varieties have been studied in~\cite{Cox95, De}. Note that~\cite[Corollary~4.7]{Cox95} contains a description of the neutral component and the component group of the automorphism groups for complete simplicial toric varieties.

There are examples of affine toric varieties with connected and not connected automorphism groups. In~\cite[Lemma 4]{Shafarevich} and~\cite[Theorem 6]{Popov}, it is shown that the automorphism group of the $ n $-dimensional affine space is connected for any positive integer number $ n $, i.e., $ \mathrm{Aut}(\mathbb{A}^n) = \mathrm{Aut}(\mathbb{A}^n)^0 $. An example of an affine toric variety with a not connected automorphism group is the algebraic torus $ T = (\mathbb{K}^{\times})^n $. It is well known that the automorphism group of the torus $ T $ is isomorphic to $ \mathrm{GL}_n(\mathbb{Z}) \rightthreetimes (\mathbb{K}^{\times})^n $ and is not connected. These considerations lead naturally to the question whether the automorphism group of an affine toric variety is connected.

The aim of this paper is to prove a criterion for the automorphism group of an affine toric variety to be connected. It is shown that the automorphism group of a degenerate affine toric variety is not connected, and the automorphism group of a non-degenerate affine toric variety is connected if and only if there are no non-trivial automorphisms of the divisor class group such that they permute the classes of prime divisors that are invariant under the action of the acting torus.

Necessary definitions are given in Section~\ref{sectprelim}. The criterion for the automorphism group of an affine toric variety to be connected is proved in Section~\ref{sectcrit}, see Theorem~\ref{crit} and Corollary~\ref{sledcompgroup}. Section~\ref{sectcompgroup} describes the component group of the automorphism group of a non-degenerate affine toric variety. It is shown that the component group is finite. It is remarkable that the description of the component group of the automorphism group is analogous to the description of the component group in the case of complete simplicial toric varieties. Section~\ref{sectsurf} contains the application of the connectedness criterion to the case of toric surfaces, see Proposition~\ref{critsurf}. Section~\ref{sectex} is dedicated to examples illustrating the obtained results.

The author is grateful to Ivan Arzhantsev and to Sergey Gaifullin for constant attention to this work. The author is a winner of the Competition of mathematical projects carried out by young researchers, performed by Leonhard Euler International Mathematical Institute in Saint Petersburg, and would like to thank its sponsors and jury.

\section{Preliminaries} \label{sectprelim}
\subsection{Toric varieties}
\label{subsecttor}

Let us recall some facts on toric varieties. More detailed information and proofs can be found in~\cite{CLS, Fulton}. A normal irreducible algebraic variety $ X $ is called \textit{toric} if it contains an algebraic torus $ T = (\mathbb{K}^{\times})^n $ as a dense open subset in the Zariski topology, and the action of the torus on itself can be extended to a regular action of $T$ on the entire variety $ X $.

Let $ X $ be an affine toric variety with an acting torus $ T $. By $ N $ we denote the lattice of one-parameter subgroups $ \lambda: \mathbb{K}^{\times} \to T $, and let $ M = \mathrm{Hom}(N, \mathbb{Z}) $ be the dual lattice. We associate the lattice $ M $ with the lattice of characters $ \chi: T \to \mathbb{K}^{\times} $, where the pairing $ N \times M \to \mathbb{Z} $ is given by the rule $$ (\lambda, \chi) \to \langle\lambda, \chi\rangle,\text{ where }c^{\langle\lambda, \chi\rangle} = \chi(\lambda(c))\text{  for  }c \in \mathbb{K}^{\times}. $$

Recall the correspondence between affine toric varieties and rational polyhedral cones. Consider a polyhedral cone $ \sigma $ in the rational vector space $ N_{\mathbb{Q}} = N \otimes_{\mathbb{Z}} \mathbb{Q} $. The dual cone $ \sigma^{\vee} $ in the space $ M_{\mathbb{Q}} = M \otimes_{\mathbb{Z}} \mathbb{Q} $ is defined as $$ \sigma^{\vee} = \{ m \in M_{\mathbb{Q}} \mid \langle u, m \rangle \geq 0 \ \forall u \in \sigma \}. $$ The variety $ X_{\sigma} = \mathrm{Spec}(\mathbb{K}[\sigma^{\vee} \cap M]) $ is toric and any affine toric variety can be constructed in such way. $ T $-orbits on the variety $ X_\sigma $ correspond to faces of the cone~$ \sigma $. In particular, to each ray of the cone $ \sigma $, one can associate a prime $ T $-invariant divisor on the variety~$ X_\sigma $, which is the closure of the corresponding $ T $-orbit. A vector is called \textit{primitive} if it is the shortest integer vector on its ray. If the cone~$ \sigma $ contains~$ r $ rays with primitive vectors $ v_1, \dots, v_r $, then the corresponding prime $ T $-invariant divisors are denoted by $ D_1, \dots, D_r $.

A toric variety is called \textit{non-degenerate} if it has only constant invertible regular functions. A toric variety $ X $ is non-degenerate if and only if it cannot be represented as a direct product of some toric variety and an algebraic torus. This condition is equivalent to the cone $ \sigma $, corresponding to the variety $ X $, being full-dimensional. In Sections~\ref{subsectCl} and~\ref{subsectR} we assume~$ X $ to be a non-degenerate affine toric variety corresponding to the cone $ \sigma $.

\subsection{A divisor class group} \label{subsectCl}
Denote by $ \mathrm{WDiv}(X) $ the group of Weil divisors on a normal algebraic variety $ X $, and by $ \mathrm{PDiv}(X) $ the subgroup of principal divisors, i.e., $$ \mathrm{PDiv}(X) = \{ \mathrm{div}(f) \mid f \in \mathbb{K}(X)^{\times} \} ,$$ where $ \mathrm{div}(f) $ denotes the divisor of the rational function $ f $. The \textit{divisor class group} $ \mathrm{Cl}(X) $ of the variety $ X $ is defined as the quotient group of the group of Weil divisors by the subgroup of principal divisors:
$$ \mathrm{Cl}(X) = \mathrm{WDiv}(X) / \mathrm{PDiv}(X), $$
see~\cite[Chapter III, \S 1]{Algeom}.

Denote by $ \mathrm{WDiv}_{T}(X) $ the subgroup of Weil divisors invariant under the action of $ T $. It is freely generated by the $ T $-invariant prime divisors $ D_1, \dots, D_r $. For each element $ m \in M $, denote by $ \chi^m $ the corresponding character of the torus. According to~\cite[Theorem~4.1.3]{CLS}, there exists an exact sequence $$ M \to \mathrm{WDiv}_{T}(X) \to \mathrm{Cl}(X) \to 0, $$ where the first map is given by the rule $ m \mapsto \mathrm{div}(\chi^m) $, and the second map sends the $ T $-invariant divisor $ D $ to its class $ [D] \in \mathrm{Cl}(X) $. By~\cite[Proposition~4.1.2]{CLS}, we have $$ \mathrm{div}(\chi^m) = \sum_{i=1}^r \langle v_i, m \rangle D_i $$ under the notation of Section~\ref{subsecttor}. Thus, $$ \mathrm{Cl}(X) \simeq \langle D_1, \dots, D_r \rangle / \langle \mathrm{div}(\chi^{e_j}) \mid j=1, \dots, n \rangle=$$ $$= \langle D_1, \dots, D_r \rangle / \Bigl\langle \sum_{i=1}^r v_{ij}D_i \Big| j=1, \dots, n \Bigr\rangle ,$$ where $ v_{i1}, \dots, v_{ij} $ are the coordinates of the vector $ v_i $ in the basis of the lattice $ N $, dual to the basis $ e_1, \dots, e_n $ of the lattice $ M $. In particular, it follows that the class group of an affine toric variety is finitely generated.

Automorphisms from $ \mathrm{Aut}(X) $ act naturally on the set of prime divisors and, consequently, on the group of Weil divisors. Principal divisors are mapped to principal ones under this action, and thus, we have an action of the group $ \mathrm{Aut}(X) $ on the class group $ \mathrm{Cl}(X) $. Therefore, there exists a homomorphism $$ \widetilde{\alpha}: \mathrm{Aut}(X) \to \mathrm{Aut}(\mathrm{Cl}(X)) .$$
Given an automorphism $ \phi \in \mathrm{Aut}(X) $, the automorphism $ \widetilde{\alpha}(\phi) \in \mathrm{Aut}(\mathrm{Cl}(X)) $ acts as follows: $$ \widetilde{\alpha}(\phi): [D] \mapsto [\phi(D)] .$$
Define an antihomomorphism $$ \alpha: \mathrm{Aut}(X) \to \mathrm{Aut}(\mathrm{Cl}(X)), \phi \mapsto \widetilde{\alpha}(\phi^{-1}) .$$ Note that
\begin{equation}
    \label{eqkerakertildea}
\mathrm{Ker}\,\alpha=\mathrm{Ker}\,\widetilde{\alpha},
\end{equation}
because $$\widetilde{\alpha}(\phi)=\mathrm{id}_{\mathrm{Cl}(X)} \Longleftrightarrow
[D]=[\phi(D)]\ \forall D\in \mathrm{WDiv}(X)
\Longleftrightarrow$$
$$\Longleftrightarrow[D]=[\phi^{-1}(D)]\ \forall D\in \mathrm{WDiv}(X)
\Longleftrightarrow
\alpha(\phi)=\mathrm{id}_{\mathrm{Cl}(X)}.
$$
Moreover, 
\begin{equation}
    \label{eqimaimtildea}
    \widetilde{\alpha}(\mathrm{Aut}(X))=\alpha(\mathrm{Aut}(X)).
\end{equation}
Indeed,
$$\xi\in \widetilde{\alpha}(\mathrm{Aut}(X)) \Longleftrightarrow 
\xi^{-1}\in \widetilde{\alpha}(\mathrm{Aut}(X)) 
\Longleftrightarrow
\xi\in \alpha(\mathrm{Aut}(X)).
$$

In~\cite[Lemma 2.2]{IVBazhov}, it is proved that for a non-degenerate affine toric variety $ X $, the component $ \mathrm{Aut}(X)^0 $ is contained in the kernel of the action of the automorphism group of~$ X $ on the class group of $ X $. In Proposition~\ref{thmAut0Kera}, we show that the equality holds: $$\mathrm{Aut}(X)^0=\mathrm{Ker}\,\widetilde{\alpha}=\mathrm{Ker}\,\alpha.$$

\subsection{Cox rings} \label{subsectR} 

Cox rings were first introduced in~\cite{Cox95}. For detailed information the reader is also referred to~\cite{ADHL}.

Recall the construction of the Cox ring for a normal algebraic variety $ X $ with only constant invertible regular functions and with finitely generated divisor class group.
For a Weil divisor $ D $ on the variety $ X $, consider the vector space
$$ L(X, D) := \{ f \in \mathbb{K}(X)^{\times} \mid \mathrm{div}(f) + D \geq 0 \} \cup \{ 0 \}. $$
For a subgroup $ K \subseteq \mathrm{WDiv}(X) $, consider the $ K $-graded $ \mathbb{K}[X] $-algebra: $$ S_K := \bigoplus_{D \in K} S_D \text{, where } S_D = L(X, D). $$ Multiplication in $ S_K $ is defined on homogeneous elements as follows. If $ f_1 \in S_{D_1} $ and $ f_2 \in S_{D_2} $, then their product in $ S_K $ is the product $ f_1f_2 $ in $ \mathbb{K}(X) $, considered as an element of $ S_{D_1 + D_2} $. For arbitrary elements of $ S_K $, multiplication is defined by distributivity.

In the group of Weil divisors, one can choose a free finitely generated subgroup~$ K $, which surjectively maps onto the divisor class group under the factoring by the subgroup of principal divisors. Consider the group homomorphism $$ \chi: K \cap \mathrm{PDiv}(X) \to \mathbb{K}(X)^{\times} ,$$ satisfying the rule $$ \mathrm{div}(\chi(E)) = E .$$ Denote by $ I $ the ideal of $ S_K $, generated by elements $ 1 - \chi(E) $ for all Weil divisors $ {E \in K \cap \mathrm{PDiv}(X) }$, where $ 1 $ is a homogeneous element of degree $ 0 $, and the element $ \chi(E) $ is homogeneous and has degree $ -E $.

The \textit{Cox ring} is defined as the quotient ring $ R(X) := S_K / I $. This ring is graded by the divisor class group of $ X $: $$ R(X) = \bigoplus_{u \in \mathrm{Cl}(X)} R(X)_u ,$$ with $ R(X)_0 = \mathbb{K}[X] $. It is known that the Cox ring of the variety $ X $ does not depend on the choice of a subgroup $ K $ and a homomorphism $ \chi $ up to an isomorphism of $ \mathrm{Cl}(X) $-graded rings.

\textit{The quasi-torus of N\'eron~--- Severi} for the variety $ X $ is a quasi-torus $ N(X) $, whose character group is isomorphic to $ \mathrm{Cl}(X) $. The quasi-torus $ N(X) $ acts on $ R(X) $ by automorphisms, and $ R(X)_u $ are weight subspaces for this action, that is, $ N $ acts on $ R(X)_u $ by multiplication by the corresponding character. Thus, under the action of elements from $ N(X) $, the $ \mathrm{Cl}(X) $-homogeneous components of $ R(X) $ are invariant.

In~\cite[Theorem 5.1]{IVSA10}, it is proved that for an irreducible normal affine variety with a finitely generated divisor class group and with only constant invertible regular functions, there exists an exact sequence
\begin{equation}
\label{exactsequence}
    1\to N(X) \to \widetilde{\mathrm{Aut}}(R(X)) \xrightarrow{\beta} \mathrm{Aut}(X) \to 1,
\end{equation}
where $ \widetilde{\mathrm{Aut}}(R(X)) $ denotes the set of the automorphisms of the Cox ring that normalize the $\mathrm{Cl}(X)$-grading:
$$ \widetilde{\mathrm{Aut}}(R(X)) := \{ \phi \in \mathrm{Aut}(R(X)) \mid \exists \phi_0 \in \mathrm{Aut}(\mathrm{Cl}(X)):$$ $$ \phi(R(X)_u) = R(X)_{\phi_0(u)} \ \forall u \in \mathrm{Cl}(X) \}. $$
Since automorphisms from $ \widetilde{\mathrm{Aut}}(R(X)) $ normalize the $\mathrm{Cl}(X)$-grading, the component $R(X)_0$ is an invariant subset for any $ \psi \in \widetilde{\mathrm{Aut}}(R(X)) $. Therefore, the restriction $ \psi|_{R(X)_0} $ is well defined. The antihomomorphism $ \beta $ is defined by the following rule: $ \beta(\psi) = \phi $, if $$ \psi|_{R(X)_0} = \psi|_{\mathbb{K}[X]} = \phi^*, $$ that is $$ \psi|_{R(X)_0}(f)(x) = f(\phi(x)) $$ for any $ x \in X, f \in \mathbb{K}[X] $.

By definition of the group $ \widetilde{\mathrm{Aut}}(R(X)) $, an automorphism of the Cox ring, that normalizes the $ \mathrm{Cl}(X) $-grading, can be associated with an automorphism of the class group. Thus, we have a group homomorphism $$ \gamma: \widetilde{\mathrm{Aut}}(R(X)) \to \mathrm{Aut}(\mathrm{Cl}(X)) .$$

In~\cite{Cox95} it is proved that for a non-degenerate toric variety $ X = X_\sigma $, the Cox ring is isomorphic to the polynomial ring in $ r $ variables over the field $ \mathbb{K} $, where $ r $ denotes the number of the rays of the cone $ \sigma $, $$ R(X) = \mathbb{K}[T_1, \dots, T_r], $$ where $ T_i $ are homogeneous with respect to the $ \mathrm{Cl}(X) $-grading and $ \mathrm{deg}(T_i) = [D_i] $.

\section{Criterion for the automorphism group to be connected} \label{sectcrit}

\begin{propos}
\label{proposdeg}
Let $X$ be a degenerate affine toric variety with an acting torus~$ T $. Then the automorphism group of $X$ is not connected.

\end{propos}
\begin{proof}

Since $ X $ is a degenerate affine toric variety, we have that $ X $ can be decomposed into the direct product $ X = Y \times \widetilde{T} $, and for the torus $ T $ it holds that $ T = \overline{T} \times \widetilde{T} $, where $ \overline{T} $ and $ \widetilde{T} $ are algebraic tori, and $ Y $ is a non-degenerate affine toric variety with an acting torus $ \overline{T} $. Thus, we have the following equality:

\begin{equation}
    \label{eqdeg}
    \mathbb{K}[X]=\mathbb{K}[Y] \otimes \mathbb{K}[\widetilde{T}]=\mathbb{K}[Y] \otimes \mathbb{K}[t_1,t_1^{-1},\dots,t_q,t_q^{-1}],
\end{equation}
where $ t_1, \dots, t_q $ are the coordinate functions on the torus $ \widetilde{T} $. Take an automorphism $ {\phi \in \mathrm{Aut}(X) }$. Then $ \phi^* $ is an automorphism of the algebra~(\ref{eqdeg}). We show that $ \phi^* $ acts on the coordinate functions $ y_1, \dots, y_p $ of the variety $ Y $ and $ t_1, \dots, t_q $ of the torus $ \widetilde{T} $ as follows:
\begin{equation}
\label{eqdeg2}
    \phi^*:\ y_i\mapsto \phi^*(y_i),\ i=1,\dots,p;\  t_j\mapsto \nu^* (t_j),\ j=1,\dots,q
\end{equation}
for some automorphism $ \nu^* $ of the algebra $ \mathbb{K}[\widetilde{T}] $. Indeed, under an action of any automorphism, invertible functions, in particular, $ t_i $, map to invertible ones, and the algebra $ \mathbb{K}[Y] $ does not contain invertible functions except constants. Since the invertible elements of the algebra~(\ref{eqdeg}) are Laurent monomials of the variables $ t_1, \dots, t_q $, elements $ \nu^*(t_j) $ do not depend on $ y_1, \dots, y_p $.

Suppose that the group $ \mathrm{Aut}(X) $ is connected. Let us show that the connectedness of $ \mathrm{Aut}(X) $ implies the connectedness of $ \mathrm{Aut}(\widetilde{T}) $, this gives a contradiction. Choose any automorphism $ \psi \in \mathrm{Aut}(\widetilde{T}) $ and construct an automorphism $ \phi \in \mathrm{Aut}(X) $ such that for $ \phi^* $ the following holds: $$ \phi^*: y_i \mapsto y_i, i = 1, \dots, p;\ t_j \mapsto \psi^*(t_j), j = 1, \dots, q. $$ From the connectedness of $ \mathrm{Aut}(X) $, it follows that $ \phi $ can be included in some algebraic family $ \{ \phi_s \}_{s \in S} $, containing the identity automorphism. Using~(\ref{eqdeg2}) we get that each automorphism~$ \phi_s^* $ has the form $$ \phi_s^*: y_i \mapsto \phi_s^*(y_i), i = 1, \dots, p;\ t_j \mapsto \nu_s^*(t_j), j = 1, \dots, q, $$ where $ \nu_s^* $ is an automorphism of the algebra $ \mathbb{K}[\widetilde{T}] $. Therefore, the family $ \{ \nu_s \}_{s \in S} $ is an algebraic family containing the automorphism $ \psi $ and the identity automorphism of the torus $ \widetilde{T} $. Hence Proposition~\ref{proposdeg} is proved. \end{proof}

Further, assume that $X$ is a non-degenerate affine toric variety. For such varieties the Cox ring $ R(X) $ is well defined; see Section~\ref{sectprelim}. Recall that antihomomorphisms
$$\alpha:\ \mathrm{Aut}(X)\to\mathrm{Aut}(\mathrm{Cl}(X)),$$ $$\beta:\ \widetilde{\mathrm{Aut}}(R(X))\to \mathrm{Aut}(X)$$ and homomorphism $$\gamma:\ \widetilde{\mathrm{Aut}}(R(X))\to \mathrm{Aut}(\mathrm{Cl}(X))$$
were introduced in the same section. 

\begin{lemma}
\label{lemdiagram}
For the maps $ \alpha, \beta $, and $ \gamma $, the equality $ \alpha \circ \beta = \gamma $ holds.
\end{lemma}
\begin{proof}
Let $ \phi \in \mathrm{Aut}(X) $. Following~\cite{IVSA10}, we construct an automorphism from $ \beta^{-1}(\phi) $. Let us show that $ \phi^* $ extends to a map
\begin{equation}
\label{eqlemdiagram}
    \phi^*:\ L(X,D)\to L(X,\phi^{-1}(D))
\end{equation}
for any Weil divisor $ D \in \mathrm{WDiv}(X) $. Indeed, $$ (\phi^*(f))(\phi^{-1}(x)) = f(x) $$ for any $ x \in X $, $ f \in \mathbb{K}(X)^{\times} $. Therefore,
\begin{equation}
\label{eqlemdiagdivphistar}
    \mathrm{div}(\phi^*(f))=\phi^{-1}(\mathrm{div}(f)),
\end{equation}
and it holds that
$$f\in L(X,D) \Longleftrightarrow
\mathrm{div}(f)+D\geq 0
\Longleftrightarrow
\phi^{-1}(\mathrm{div}(f))+\phi^{-1}(D)\geq 0
\overset{\eqref{eqlemdiagdivphistar}}{\Longleftrightarrow} $$ $$
\overset{\eqref{eqlemdiagdivphistar}}{\Longleftrightarrow} 
\mathrm{div}(\phi^*(f))+\phi^{-1}(D)\geq 0
\Longleftrightarrow
\phi^*(f)\in L(X,\phi^{-1}(D))$$ 
for any $ f \in \mathbb{K}(X)^{\times} $, $ D \in \mathrm{WDiv}(X) $. Thus,~(\ref{eqlemdiagram}) is proved. Consequently, $$ \phi^*: S_K = \bigoplus_{D \in K} L(X,D) \to S_{\phi^{-1}(K)} = \bigoplus_{D \in \phi^{-1}(K)} L(X,D) .$$
Let us prove that $ \phi^* $ defines a correct automorphism of $ R(X) $. An image of $ 1 - \chi(E) $ under the map $ \phi^* $ is $ 1 - \phi^*(\chi(E)) $, where $ 1 $ is a homogeneous element of degree $ 0 $ and $ \phi^*(\chi(E)) $ has degree $ -\phi^{-1}(E) $ by~(\ref{eqlemdiagram}). Define the group homomorphism $$ \chi' = \phi^* \circ \chi \circ \phi:\ \mathrm{PDiv}(X) \cap \phi^{-1}(K) \to \mathbb{K}(X)^{\times} .$$ The constructed homomorphism $\chi'$ satisfies the equality $ \mathrm{div}(\chi'(D)) = D $ for any Weil divisor $ D \in \mathrm{PDiv}(X) \cap \phi^{-1}(K) $. Indeed, $$ \mathrm{div}(\chi'(D)) = \mathrm{div}(\phi^* \circ \chi \circ \phi(D)) \overset{\eqref{eqlemdiagdivphistar}}{=} \phi^{-1}(\mathrm{div}(\chi \circ \phi(D))) = \phi^{-1}(\phi(D)) = D $$ for any principal Weil divisor $ D $ from $ \phi^{-1}(K) $. Thus, $ \phi^* $ maps the ideal $ I $ to the ideal of the ring $ S_{\phi^{-1}(K)} $, generated by elements $ 1 - \chi'(D) $ for all Weil divisors $ D $ from 
$ \mathrm{PDiv}(X) \cap \phi^{-1}(K) $, where $ 1 $ is a homogeneous element of degree $ 0 $, and the element $ \chi'(D) $ is homogeneous and has degree $ -D $. Therefore, $ \phi^* $ is a homomorphism that maps the Cox ring constructed using the homomorphism $ \chi $ and subgroup $ K $, to the Cox ring constructed using the homomorphism $ \chi' $ and subgroup $ \phi^{-1}(K) $. The Cox ring does not depend on the choice of a homomorphism and a subgroup in the group of Weil divisors, satisfying the conditions from Section~\ref{subsectR}. Moreover, a homomorphism of the Cox ring to itself obtained by the same construction from the automorphism $ \phi^{-1} $ is an inverse to $ \phi^* $. Hence, $\phi^*$ is an automorphism of the Cox ring.

Therefore, it is shown that the set $ \beta^{-1}(\phi) $ contains the automorphism $ \phi^* $, which maps the Cox ring component of degree $ [D] $ to the component of degree $ [\phi^{-1}(D)] $ with respect to the $ \mathrm{Cl}(X) $-grading. It remains to note that from the exactness of the sequence~(\ref{exactsequence}), any other automorphism from $ \beta^{-1}(\phi) $ differs from $ \phi^* $ by an automorphism corresponding to the action of some element from the N\`eron~--- Severi quasi-torus. Homogeneous components of $R(X)$ are preserved under the action of the quasi-torus $ N(X) $. Consequently, for any automorphism $ \psi^* \in \widetilde{\mathrm{Aut}}(R(X)) $ and for any Weil divisor $ D $, it holds that $$ \gamma(\psi^*) : [D] \mapsto [\beta(\psi^*)^{-1}(D)] = \alpha \circ \beta(\psi^*)([D]). $$ Thus, Lemma~\ref{lemdiagram} is proved.\end{proof}

Consider the kernel of the homomorphism $ \gamma $. It consists precisely of those automorphisms of the Cox ring that preserve the $\mathrm{Cl}(X)$-grading. With any element $ g^* \in \mathrm{Ker}\,\gamma $, one can associate an automorphism
$${g\in\mathrm{Aut}(\mathrm{Spec}(R(X))) = \mathrm{Aut}(\mathbb{A}^r)}.$$ We define $$G:=\{ g\in \mathrm{Aut}(\mathbb{A}^r)\ |\  g^*\in\mathrm{Ker}\,\gamma\}\subseteq \mathrm{Aut}(\mathbb{A}^r).$$ 

\begin{lemma}
    The subgroup $ G $ is a connected subgroup of $ \mathrm{Aut}(\mathbb{A}^r) $.
    \label{lemkergamma}
\end{lemma}
\begin{proof}

Using methods of~\cite[Lemma 4]{Shafarevich} and~\cite[Theorem 6]{Popov}, we show that each automorphism from $ G $ is a composition of automorphisms from the subgroups $ A $ and~$ H $, where $ A $ and $ H $ are connected subgroups in $ G $. This implies the connectedness of $ G $ in $ \mathrm{Aut}(\mathbb{A}^r) $.

Recall that the Cox ring of the toric variety $ X $ is a polynomial ring with a $ \mathrm{Cl}(X) $-grading: $$ R(X)=\mathbb{K}[T_1,\dots,T_r], $$ where $ \mathrm{deg}(T_i)=[D_i]\in\mathrm{Cl}(X) $.

For any automorphism $ \phi^* \in\mathrm{Aut}(R(X)) $, let us define a homomorphism $ l(\phi^*) $ of the algebra $ R(X) $ into itself, constructed as follows. Suppose that $ \phi^* $ acts on the variables according to the formula
\begin{equation}
\label{eqlemG}
   \phi^*:\ T_i \mapsto F_{i0}+F_{i1}(T_{1},\dots,T_{r})+\dots+F_{im}(T_{1},\dots,T_{r}), 
\end{equation}
where $ F_{ij}(T_{1},\dots,T_{r}) $ is a form of degree $ j $ in $ T_1,\dots,T_r $. Then we define $$ l(\phi^*):\ T_i\mapsto F_{i0}+F_{i1}(T_{1},\dots,T_{r}) $$ and extend it to $ R(X)=\mathbb{K}[T_1,\dots,T_r] $ by linearity and multiplicativity.

Note that $$ l(\phi^* \circ \psi^*)=l(\phi^*)\circ l(\psi^*) $$ for any $ \phi^*,\psi^*\in\mathrm{Aut}(R(X)) $. Hence, we have $$ \mathrm{id}_{R(X)}=l(\phi^* \circ (\phi^*)^{-1})=l(\phi^*)\circ l((\phi^*)^{-1}), $$ which means $ l(\phi^*)^{-1}=l((\phi^*)^{-1}) $, and $ l(\phi^*) $ is invertible. Therefore, $ l(\phi^*) $ is an automorphism of the algebra $ R(X) $ for any automorphism $ \phi^*\in \mathrm{Aut}(R(X)) $.

Further, consider the set $$ A^*:= \{l(\phi^*)\ |\ \phi^*\in\mathrm{Ker}\,\gamma \} .$$ Let us check that if an automorphism $ \phi^* $ normalizes the $ \mathrm{Cl}(X) $-grading, then $ l(\phi^*) $ normalizes the $ \mathrm{Cl}(X) $-grading as well. To do this we show that for any element $ G \in R(X) $ we have
\begin{equation}
    \label{eqlemG2}
    \mathrm{deg}(\phi^*(G))=\mathrm{deg}(l(\phi^*)(G)).
\end{equation}
Then we can take $ \gamma(\phi^*) $ as $ \phi_0 $ for $ l(\phi^*) $ in the definition of $ \widetilde{\mathrm{Aut}}(X) $. For $ T_1,\dots,T_r $, equality~(\ref{eqlemG2}) is satisfied because all terms in expression~(\ref{eqlemG}) are homogeneous and have degree $ \gamma(\phi^*)([D_i]) $. For products of homogeneous elements and sums of homogeneous elements of the same degree, equality~(\ref{eqlemG2}) is obtained by the linearity and multiplicativity of $ l(\phi^*) $.

Moreover, $ \gamma(\phi^*)=\gamma(l(\phi^*)) $. Therefore, for any $ \phi^*\in\mathrm{Ker}\,\gamma $, the automorphism $ l(\phi^*) $ is also contained in the kernel of $ \gamma $. It is straightforward to verify that $ A^* $ is a subgroup of~$ \mathrm{Ker}\,\gamma $.

Consider the subgroup $$ A:=\{ a\in G\ |\ a^*\in A^*\}\subseteq G .$$

Let us prove that the subgroup $ A $ is connected. Suppose that there are exactly $ k $ distinct elements $ d_1,\dots,d_k $ among $ [D_1],\dots,[D_r] $ and for any $ i=1,\dots,k $ there are exactly $ n_i $ variables $ T_j $ with degree $ d_i $, i.e., $ r=n_1+\dots+n_k $. Let us introduce a new notation for the indices of $ T_j $. Let 
$$\mathrm{deg}(T_{11})=\dots=\mathrm{deg}(T_{1n_1})=d_1,$$
$$\dots$$
$$\mathrm{deg}(T_{k1})=\dots=\mathrm{deg}(T_{kn_k})=d_k.$$
If there is no zero element among the elements $ d_1,\dots,d_k $, then the group
\begin{equation}
    \label{eqlemkergamma}
    A\simeq A^* = \mathrm{GL}_{n_1}(\mathbb{K})\times\dots\times \mathrm{GL}_{n_k}(\mathbb{K})
\end{equation}
is connected. If $ d_i=0\in\mathrm{Cl}(X) $, then the corresponding factor $ \mathrm{GL}_{n_i}(\mathbb{K}) $  in the product~(\ref{eqlemkergamma}) is replaced by $ \mathrm{GL}_{n_i}(\mathbb{K}) \rightthreetimes \mathbb{K}^{n_i} $. In this case, $ A $ is also connected.

By definition, put $$H^*:=\{l(\phi^*)^{-1}\cdot \phi^*\ |\ \phi^*\in \mathrm{Ker}\,\gamma\}=\{\phi^*\in\mathrm{Ker}\,\gamma\ |\ l(\phi^*)=\mathrm{id}_{R(X)}\}.$$ Note that $H^*$ is a subgroup of $\mathrm{Ker}\,\gamma$.

Let us consider the subgroup $$ H:=\{h\in G\ |\ h^*\in H^* \} $$ and prove that it is connected. Take any automorphism $ h\in H $ and the corresponding automorphism $ h^*\in H^* $. The linear part of the automorphism $ h^* $ is the identity automorphism. Therefore, $ h^* $ is of the form $$ h^*(T_i)=T_i+ \sum_{j=2}^{m_i} h_{ij},\ i=1,\dots,r $$ for some numbers $ m_i\in\mathbb{N}_{\geq 2} $, where $ h_{ij} $ are either zero or homogeneous forms of degree~$ j $ in $ T_1,\dots, T_r $. Here, homogeneity is understood in the sense of the standard grading $ \mathrm{deg}(T_j)=1\in \mathbb{Z} $ for any $ j=1,\dots ,r $.

For any element $ t\in \mathbb{K}^{\times} $, denote by $ \xi_{t}^*\in A^* $ the automorphism that acts on the variables as follows: $$ \xi_{t}^*(T_i)=tT_i,\ i=1,\dots,r. $$
Let $$ h^*_t:=(\xi_{t}^*)^{-1}\circ h^* \circ \xi_{t}^*, $$ then $$ h^*_t(T_i)=T_i+\sum_{j=2}^{m_i}t^{j-1}h_{ij},\ i=1,\dots,r. $$
Define $ h^*_0:=\mathrm{id}_{R(X)} $, and then $ \{h_t\}_{t\in\mathbb{K}}\subseteq H $ is an algebraic family containing $ h=h_1 $ and $ \mathrm{id}_{\mathbb{A}^r}=h_0 $. Consequently, the subgroup $ H $ is connected.

It remains to note that for any $ \phi^*\in \mathrm{Ker}\,\gamma $, it holds that $$ \phi^*=l(\phi^*)\circ l(\phi^*)^{-1}\circ \phi^*=a^*\circ h^*, $$ where $ a^*=l(\phi^*)\in A^* $, and $ h^*=l(\phi^*)^{-1}\circ \phi^*\in H^*. $ Thus, for any $ \phi\in G $, there exists a decomposition $ \phi=h\circ a $, where $ h\in H$ and $ a\in A $. From the connectedness of the subgroups~$ A $ and $ H $, the connectedness of the group $ G $ follows. Hence, Lemma~\ref{lemkergamma} is proved.\end{proof}

Note that the antihomomorphism $ \beta $ surjectively maps the kernel of $ \gamma $ onto the kernel of~$ \alpha $. Indeed, if an element $ f $ is contained in the kernel of $ \alpha $, then its preimage with respect to $ \beta $ exists and lies in the group $ \widetilde{\mathrm{Aut}}(R(X)) $ due to the surjectivity of~$ \beta $. Then, by Lemma~\ref{lemdiagram}, it follows that $ \gamma(\beta^{-1}(f))=\alpha(f)=\mathrm{id}_{\mathrm{Cl}(X)} $ and the element $ \beta^{-1}(f) $ is contained in $ \mathrm{Ker}\,\gamma $.

\begin{propos}
\label{thmAut0Kera}
    For a non-degenerate affine toric variety $ X $, the equality $$ \mathrm{Aut}(X)^0=\mathrm{Ker}(\mathrm{Aut}(X)\curvearrowright \mathrm{Cl}(X)) $$ holds.
\end{propos}
\begin{proof}

Recall that $$ \mathrm{Ker}(\mathrm{Aut}(X)\curvearrowright \mathrm{Cl}(X))=\mathrm{Ker}\,\widetilde{\alpha} \overset{\eqref{eqkerakertildea}}{=} \mathrm{Ker}\,\alpha .$$
Let us show that $ \mathrm{Ker}\,\alpha $ is a connected subgroup of $ \mathrm{Aut}(X) $. Take any automorphism $$ \phi \in \mathrm{Ker}\,\alpha\subseteq \mathrm{Aut}(X) .$$ By the surjectivity of $ \beta $, there exists an automorphism $$ \psi^*\in\mathrm{Ker}\,\gamma\subseteq \widetilde{\mathrm{Aut}}(R(X))$$ such that $ \beta(\psi^*)=\phi $. Automorphism $ \psi^* $ corresponds to an automorphism $$ \psi \in G\subseteq \mathrm{Aut}(\mathrm{Spec}(R(X)))=\mathrm{Aut}(\mathbb{A}^r).$$ The subgroup $ G $ is connected in $ \mathrm{Aut}(\mathbb{A}^r) $ by Lemma~\ref{lemkergamma}. Whence, for some irreducible affine algebraic variety $ S $, there exists an algebraic family $ \Psi=\{\psi_{s}\}_{s\in S}\subseteq G $, containing $ \psi $ and~$ \mathrm{id}_{\mathbb{A}^r} $. Since this family is algebraic, the map $$ \xi: S\times \mathbb{A}^r\to \mathbb{A}^r: (s,z)\mapsto \psi_s(z) $$ is a morphism of algebraic varieties.

For each element $ \psi_s\in \Psi $, consider $ \psi_s^*\in \mathrm{Aut}(R(X)) $. Note that $ \psi_s^* $ is contained in~$ \mathrm{Ker}\,\gamma $, because $ \psi_s\in G $ for any $ s\in S $. Therefore, the automorphism $ \psi_s^* $ preserves the $ \mathrm{Cl}(X) $-grading on $ R(X) $, and the restriction $ \psi_s^*|_{R(X)_0} $ is well defined. Denote
\begin{equation}
    \label{eqthmAut0Kera}
\phi_s^*:=\psi_s^*|_{R(X)_0}=\psi_s^*|_{\mathbb{K}[X]}\in\mathrm{Aut}(\mathbb{K}[X]).    
\end{equation}

The automorphism $ \phi_s^*\in \mathrm{Aut}(\mathbb{K}[X]) $ corresponds to an automorphism $ \phi_s$ from $\mathrm{Aut}(X) $. Moreover, $ \psi_s^*\in\mathrm{Ker}\,\gamma $ and $ \phi_s=\beta(\psi_s^*) $, hence, $ \phi_s\in \mathrm{Ker}\,\alpha $. Thus, the set $ \Phi=\{\phi_s\}_{s\in S} $ is a family in $ \mathrm{Ker}\,\alpha $, containing $ \phi $ and the identity automorphism of the variety $ X $. It remains to prove that the family $ \Phi $ is algebraic.

For the morphism $ \xi $, defined above, consider the homomorphism $$ \xi^*: R(X)\to \mathbb{K}[S]\otimes R(X) .$$ Note that
\begin{equation}
    \label{eqthmAut0Kera2}
(\xi^*(f))(s,z)=f(\xi(s,z))=f(\psi_s(z))=(\psi_s^*(f))(z)
\end{equation}
for any $ s\in S $, $ f\in R(X) $, and $ z\in \mathrm{Spec}(R(X))=\mathbb{A}^r $.

Let $ f\in R(X)_0 $. Then by~(\ref{eqthmAut0Kera2}) we have $$ \xi^*(f)\in \mathbb{K}[S]\otimes R(X)_0 ,$$ as $ \psi_s^*\in \mathrm{Ker}\,\gamma $, and therefore, $ \psi_s^*(f)\in R(X)_0 $.

Hence, the homomorphism $$ \zeta^*:=\xi^*|_{R(X)_0}: R(X)_0\to \mathbb{K}[S]\otimes R(X)_0 $$ is well defined.
Taking into account that $ R(X)_0=\mathbb{K}[X] $, we obtain that the algebra homomorphism $ \zeta^* $ corresponds to the morphism $$ \zeta: S\times X\to X .$$ We show that $\zeta$ is a required morphism, i.e., $ \zeta(s,x)=\phi_s(x) $. For any elements $ s\in S $, $ f\in\mathbb{K}[X] $, and $ x\in X $, it holds that
$$f(\zeta(s,x))=(\zeta^*(f))(s,x)=(\xi^*(f))(s,x) \overset{\eqref{eqthmAut0Kera2}}{=} (\psi_s^*(f))(x) \overset{\eqref{eqthmAut0Kera}}{=} (\phi_s^*(f))(x)=f(\phi_s(x)).$$

Therefore, the maximal ideals in $ \mathbb{K}[X] $, corresponding to the points $ \zeta(s,x) $ and $ \phi_s(x) $, coincide, i.e., $ \mathfrak{m}_{\zeta(s,x)}=\mathfrak{m}_{\phi_s(x)} $, where $ \mathfrak{m}_x=\{f\in\mathbb{K}[X]\ |\ f(x)=0\} $. Consequently, the equality $ \zeta(s,x)=\phi_s(x) $ holds, and the morphism $ \zeta $ is the required one.

Thus, we obtain that any automorphism $ \phi\in\mathrm{Ker}\,\alpha $ can be included in some algebraic family $ \Phi\subseteq \mathrm{Ker}\,\alpha $, containing the identity automorphism of the variety~$ X $. Therefore, $ \mathrm{Ker}\,\alpha $ is a connected subgroup of $ \mathrm{Aut}(X) $. Hence, we have an inclusion $$\mathrm{Ker}\,\alpha=\mathrm{Ker}(\mathrm{Aut}(X)\curvearrowright \mathrm{Cl}(X)) \subseteq \mathrm{Aut}(X)^0 .$$

The reverse inclusion follows from~\cite[Lemma 2.2]{IVBazhov}. Proposition~\ref{thmAut0Kera} is proved.\end{proof}

Therefore, the diagram in Figure 1 is commutative.
\\

\tikzset{every picture/.style={line width=0.75pt}} 

\begin{center}

\begin{tikzpicture}[x=0.75pt,y=0.75pt,yscale=-1,xscale=1]

\draw    (201,28) -- (404.61,28) ;
\draw [shift={(404.61,28)}, rotate = 180] [color={rgb, 255:red, 0; green, 0; blue, 0 }  ][line width=0.75]    (12.35,-3.43) .. controls (9.56,-1.61) and (7.01,-0.47) .. (4.7,0) .. controls (7.01,0.47) and (9.56,1.61) .. (12.35,3.43)(7.65,-3.43) .. controls (4.86,-1.61) and (2.31,-0.47) .. (0,0) .. controls (2.31,0.47) and (4.86,1.61) .. (7.65,3.43)   ;
\draw    (206.11,74.07) -- (406.11,74.07) ;
\draw [shift={(406.11,74.07)}, rotate = 180] [color={rgb, 255:red, 0; green, 0; blue, 0 }  ][line width=0.75]    (12.35,-3.43) .. controls (9.56,-1.61) and (7.01,-0.47) .. (4.7,0) .. controls (7.01,0.47) and (9.56,1.61) .. (12.35,3.43)(7.65,-3.43) .. controls (4.86,-1.61) and (2.31,-0.47) .. (0,0) .. controls (2.31,0.47) and (4.86,1.61) .. (7.65,3.43)   ;
\draw    (168.5,86.5) -- (251.5,147.88) ;
\draw [shift={(253.11,149.07)}, rotate = 216.48] [color={rgb, 255:red, 0; green, 0; blue, 0 }  ][line width=0.75]    (7.65,-2.3) .. controls (4.86,-0.97) and (2.31,-0.21) .. (0,0) .. controls (2.31,0.21) and (4.86,0.98) .. (7.65,2.3)   ;
\draw    (430,87.5) -- (345.24,147.41) ;
\draw [shift={(343.61,148.57)}, rotate = 324.74] [color={rgb, 255:red, 0; green, 0; blue, 0 }  ][line width=0.75]    (7.65,-2.3) .. controls (4.86,-0.97) and (2.31,-0.21) .. (0,0) .. controls (2.31,0.21) and (4.86,0.98) .. (7.65,2.3)   ;

\draw (148.11,16.97) node [anchor=north west][inner sep=0.75pt]    {$\mathrm{Ker} \ \gamma $};
\draw (408.11,16.97) node [anchor=north west][inner sep=0.75pt]    {$\mathrm{Ker} \ \alpha =\mathrm{Aut}( X)^{0}$};
\draw (128.11,56.97) node [anchor=north west][inner sep=0.75pt]    {$\widetilde{\mathrm{Aut}}( R( X))$};
\draw (411.11,63.97) node [anchor=north west][inner sep=0.75pt]    {$\mathrm{Aut}( X)$};
\draw (257.61,145.47) node [anchor=north west][inner sep=0.75pt]    {$\mathrm{Aut}(\mathrm{Cl}( X))$};
\draw (175.58,36.99) node [anchor=north west][inner sep=0.75pt]  [rotate=-90]  {$\subseteq $};
\draw (438.08,36.99) node [anchor=north west][inner sep=0.75pt]  [rotate=-90]  {$\subseteq $};
\draw (395,115.9) node [anchor=north west][inner sep=0.75pt]    {$\alpha $};
\draw (294.5,53.4) node [anchor=north west][inner sep=0.75pt]    {$\beta $};
\draw (189,112.4) node [anchor=north west][inner sep=0.75pt]    {$\gamma $};
\draw (305.5,178.9) node [anchor=north west][inner sep=0.75pt]    {Figure\ 1.};

\end{tikzpicture}

\end{center}

Let us use the description of the neutral component obtained in Proposition~\ref{thmAut0Kera} to prove the criterion for the automorphism group of a
non-degenerate affine toric variety to be connected.
\begin{theorem}
\label{crit}
    Let $ X $ be a non-degenerate affine toric variety with an acting torus $ T=(\mathbb{K}^{\times})^n $. Then the following conditions are equivalent:
    \begin{enumerate}
        \item the automorphism group of $ X $ is connected;
        \item automorphisms of the Cox ring that normalize the $\mathrm{Cl}(X)$-grading preserve this grading, i.e., $ \widetilde{\mathrm{Aut}}(R(X))=\mathrm{Ker}\,\gamma $;
        \item there is no linear operator $ L\in\mathrm{GL}_n(\mathbb{Z}) $, $ L(\sigma)=\sigma $ such that $ L(v_i)=v_j $, but ${ [D_i]\neq [D_j] }$, where $ v_i $ is the primitive vector on the $ i $-th ray of the cone $ \sigma $, and $ [D_i] $ is the class of the corresponding $ T $-invariant prime divisor in the divisor class group.

    \end{enumerate}
\end{theorem}
Note that another condition equivalent to the connectedness of the automorphism group is given in Corollary~\ref{sledcompgroup}.
\begin{proof}

Let us prove the equivalence of conditions (1) and (2). If $$ \widetilde{\mathrm{Aut}}(R(X))=\mathrm{Ker}\,\gamma ,$$
then by the diagram from Figure~1, we have $ \mathrm{Aut}(X)=\mathrm{Ker}\,\alpha $. By Proposition~\ref{thmAut0Kera}, it follows that $ \mathrm{Ker}\,\alpha=\mathrm{Aut}(X)^0 $. Therefore, $ \mathrm{Aut}(X)=\mathrm{Aut}(X)^0 $, and the group $ \mathrm{Aut}(X) $ is connected.
Conversely, if there exists an automorphism $$ \phi\in \widetilde{\mathrm{Aut}}(R(X)) \setminus \mathrm{Ker}\,\gamma ,$$ then $$ \beta(\phi)\in \mathrm{Aut}(X)\setminus \mathrm{Ker}\,\alpha=\mathrm{Aut}(X)\setminus \mathrm{Aut}(X)^0 ,$$ and the group $ \mathrm{Aut}(X) $ is not connected. Hence, the equivalence of conditions (1) and (2) is proved.

It remains to prove that conditions (1) and (3) are equivalent. The existence of a linear operator $ L\in\mathrm{GL}_n(\mathbb{Z}) $, $ L(\sigma)=\sigma $ such that for some natural numbers $ i,j $ it holds that $ L(v_i)=v_j $, but $ [D_i]\neq [D_j] $ is equivalent to the existence of a $ T $-equivariant automorphism $ \phi\in \mathrm{Aut}(X) $ such that for some $ i,j $ it holds that $ \phi(D_i)=D_j $, but $ [D_i]\neq [D_j] $, see~\cite[Theorem 3.3.4]{CLS}. Thus, $ \phi\not\in \mathrm{Aut}(X)^0 $, as $ \phi $ acts non-trivially on the divisor class group, and the group $ \mathrm{Aut}(X) $ is not connected.

Conversely, suppose that the automorphism group of $ X $ is not connected. Let us use Corollary~\ref{sledcompgroup}, which is proved in the next section. Since the automorphism group of $ X $ is not connected, it follows that there exists a non-trivial automorphism $ \phi $ of the group $ \mathrm{Cl}(X) $, permuting the elements $ [D_1],\dots,[D_r] $ according to some permutation $ \tau\in S_r $.

We fix some basis $ e_1,\dots,e_n $ of the lattice $ M $. Suppose that primitive vectors on the rays of the cone $ \sigma $ have coordinates
$$v_i=\begin{pmatrix}
    v_{i1} \\
    \vdots \\
    v_{in}
\end{pmatrix},\ i=1,\dots,r$$ in the basis of the vector space $ N_{\mathbb{Q}} $, dual to $ e_1,\dots,e_n $. 

Denote by $ V $ and $ V_{\tau^{-1}} $ matrices composed of the coordinates of these vectors:
$$V=\begin{pmatrix}
    v_1 & \dots & v_r
\end{pmatrix}
,\ V_{\tau^{-1}}=\begin{pmatrix}
 v_{\tau^{-1}(1)} & \dots & v_{\tau^{-1}(r)} 
\end{pmatrix}.$$
From Section~\ref{subsectCl}, it follows that for the elements $ [D_1],\dots,[D_r] $ the relations
\begin{equation}
\label{eqcrit}
    V \begin{pmatrix} [D_1] \\ \vdots \\ [D_r] \end{pmatrix} =
\begin{pmatrix}
    [\mathrm{div}(\chi^{e_1})] \\ \vdots \\ [\mathrm{div}(\chi^{e_n})]
\end{pmatrix}
    =
 \begin{pmatrix} 0 \\ \vdots \\ 0 \end{pmatrix}
\end{equation} 
hold and any other relations on the elements $ [D_1],\dots,[D_r] $ are linear combinations of the relations from~(\ref{eqcrit}), because the subgroup of $ T $-invariant principal divisors is generated by the elements $ \mathrm{div}(\chi^{e_1}),\dots,\mathrm{div}(\chi^{e_n}) $.
Since $ \phi $ is a group homomorphism we have
$$ V \begin{pmatrix} \phi([D_1]) \\ \vdots \\ \phi([D_r]) \end{pmatrix} = V
\begin{pmatrix} [D_{\tau(1)}] \\ \vdots \\ [D_{\tau(r)}] \end{pmatrix}=V_{\tau^{-1}}\begin{pmatrix} [D_1] \\ \vdots \\ [D_r] \end{pmatrix}=\begin{pmatrix} 0 \\ \vdots \\ 0 \end{pmatrix}.$$ By definition, put $$\begin{pmatrix} \widetilde{D_1} \\ \vdots \\ \widetilde{D_n} \end{pmatrix}:=
V_{\tau^{-1}}\begin{pmatrix} D_1 \\ \vdots \\ D_r \end{pmatrix}.$$ 

Divisors $ \widetilde{D_1}, \dots , \widetilde{D_n} $ are $ T $-invariant as linear combinations of $ T $-invariant ones. Moreover, they are principal since their image in the divisor class group is zero. Therefore,
$$\begin{pmatrix} \widetilde{D_1} \\ \vdots \\ \widetilde{D_n} \end{pmatrix}=
L
\begin{pmatrix}
    \mathrm{div}(\chi^{e_1}) \\ \vdots \\ \mathrm{div}(\chi^{e_n})
\end{pmatrix}$$ 
for some integer $ n\times n $-matrix $ L $. It remains to note that
$$\begin{pmatrix} \widetilde{D_1} \\ \vdots \\ \widetilde{D_n} \end{pmatrix}=V_{\tau^{-1}}\begin{pmatrix} D_1 \\ \vdots \\ D_r \end{pmatrix}=
L
\begin{pmatrix}
    \mathrm{div}(\chi^{e_1}) \\ \vdots \\ \mathrm{div}(\chi^{e_n})
\end{pmatrix}
=
LV
\begin{pmatrix}
    D_1 \\ \vdots \\ D_r
\end{pmatrix},$$ 
and elements $ D_1,\dots,D_r $ are independent. Hence, $ V_{\tau^{-1}}=LV $ and the matrix $ L $ is non-degenerate. The corresponding linear operator maps the vector $ v_i $ to the vector $ v_{\tau^{-1}(i)} $ for $ i=1,\dots,r $. Note that there exists a number $ j $ such that $$ [D_j]\neq [D_{\tau^{-1}(j)}]=\phi^{-1}([D_j]), $$ since $ \phi $ is a non-trivial automorphism of the divisor class group, and the elements $ [D_1],\dots,[D_r] $ generate the class group. Therefore, Theorem~\ref{crit} is proved.\end{proof}

\section{The component group}
\label{sectcompgroup}

Further assume that $X$ is a non-degenerate affine toric variety corresponding to a rational polyhedral cone $\sigma$. The \textit{component group} of the automorphism group of~$X$ is the quotient group $$\mathrm{Aut}(X)/\mathrm{Aut}(X)^0.$$
Denote by $\Sigma_D$ the set of maps of $\mathrm{Cl}(X)$ to itself such that they permute the elements $[D_1],\dots,[D_r]$, that is, $$\Sigma_D:=\{\phi:\mathrm{Cl}(X)\to \mathrm{Cl}(X)\ |\ \exists\tau\in S_r:\ \phi([D_i])=[D_{\tau(i)}]\}.$$ Each element of $\Sigma_D$ corresponds to at least one permutation from $S_r$ according to its action on $[D_1],\dots,[D_r]$. Moreover, the permutations corresponding to different elements of $\Sigma_D$ are distinct. Therefore, $|\Sigma_D|\leq |S_r|=r!$.

Let us describe the component group of $\mathrm{Aut}(X)$ and prove that it is finite.
\begin{theorem}
\label{thmcompgroup}
Let $X$ be a non-degenerate affine toric variety $X$. Then we have
     \begin{equation}
         \label{eqthmcompgroup}
         \mathrm{Aut}(X)/\mathrm{Aut}(X)^0\simeq \widetilde{\alpha}(\mathrm{Aut}(X))=\mathrm{Aut}(\mathrm{Cl}(X))\cap \Sigma_D.
     \end{equation}  
     In particular, 
     \begin{equation}
         \label{eqthmcompgroup2}
         |\mathrm{Aut}(X)/\mathrm{Aut}(X)^0|\leq r!,
     \end{equation}
     where $r$ denotes the number of rays of the cone $\sigma$.
\end{theorem}
\begin{proof}
    The first part of~(\ref{eqthmcompgroup}) follows from Proposition~\ref{thmAut0Kera} and the Fundamental theorem on homomorphisms:
    $$\mathrm{Aut}(X)/\mathrm{Aut}(X)^0\simeq \mathrm{Aut}(X)/\mathrm{Ker}\,\widetilde{\alpha}\simeq \widetilde{\alpha}(\mathrm{Aut}(X)).$$
    Moreover, we have
    $$\widetilde{\alpha}(\mathrm{Aut}(X)) 
\overset{\eqref{eqimaimtildea}}{=} \alpha(\mathrm{Aut}(X)).$$
    By Lemma~\ref{lemdiagram}, the diagram from Figure~1 is commutative. Hence,
    $$\alpha(\mathrm{Aut}(X))=\gamma(\widetilde{\mathrm{Aut}}(R(X))).$$
    It remains to prove that $$\gamma(\widetilde{\mathrm{Aut}}(R(X)))= \mathrm{Aut}(\mathrm{Cl}(X))\cap \Sigma_D.$$ If $\phi$ is an automorphism of $\mathrm{Cl}(X)$ that permutes elements $[D_1], \dots, [D_r]$ according to some permutation $\tau$, then its preimage under $\gamma$ contains the automorphism ${T_i\mapsto T_{\tau(i)}}$. This map is an automorphism of the Cox ring and it normalizes the $\mathrm{Cl}(X)$-grading.

    Conversely, the inclusion $\gamma(\widetilde{\mathrm{Aut}}(R(X)))\subseteq \mathrm{Aut}(\mathrm{Cl}(X))$ follows from the definition of the homomorphism $\gamma$. To show the inclusion $\gamma(\widetilde{\mathrm{Aut}}(R(X)))\subseteq \Sigma_D$, take any ${\phi\in \widetilde{\mathrm{Aut}}(R(X))}$. The Jacobian of $\phi$ is a non-zero element of the field $\mathbb{K}$. Hence, there exists a permutation $\tau\in S_r$ such that $$\frac{\partial\phi(T_1)}{\partial T_{\tau(1)}} \dots \frac{\partial\phi(T_r)}{\partial T_{\tau(r)}}$$ contains a non-zero element of the field $\mathbb{K}$ as a summand. Therefore, for each $i=1,\dots ,r$, the element $\phi(T_i)$ contains a non-zero linear term in $T_{\tau(i)}$:$$\phi(T_i)=c_i T_{\tau(i)}+\dots,\ i=1\dots,r$$ for some non-zero $c_i \in \mathbb{K}$. 
    Taking into account that $\mathrm{deg}(T_i)=[D_i]$ and that the automorphism $\phi$ maps homogeneous elements to homogeneous ones, we obtain $$\gamma(\phi)([D_i])=[D_{\tau(i)}].$$ The inclusion $\gamma(\widetilde{\mathrm{Aut}}(R(X)))\subseteq \Sigma_D$ is proved.

    The inequality~(\ref{eqthmcompgroup2}) follows from the proved~(\ref{eqthmcompgroup}) and the fact that $|\Sigma_D|\leq r!$. Theorem~\ref{thmcompgroup} is proved.
\end{proof}

\begin{remark}
Let us remark that \cite[Corollary 4.7 (v)]{Cox95} contains a description of the component group of the automorphism group of a complete simplicial toric variety. We fix necessary notation according to~\cite{Cox95}. The rays of the fan~$\Delta$ corresponding to the toric variety $X$ can be represented as the partition $\Delta_1\cup\dots\cup \Delta_s$, where $T_j$, corresponding to the rays from one $\Delta_i$, have the same $\mathrm{Cl}(X)$-degree. Denote by $\mathrm{Aut}(N,\Delta)$ the automorphism group of the lattice $N$ that preserve the fan $\Delta$. Consider the subgroups $\Sigma_{\Delta_i}$ in the group $\mathrm{Aut}(N,\Delta)$, consisting of automorphisms that permute elements in $\Delta_i$ and do not change other elements. For a complete simplicial toric variety $X$ the following holds:
$$\mathrm{Aut}(X)/\mathrm{Aut}(X)^0\simeq \mathrm{Aut}(N,\Delta) / \prod_{i=1}^s \Sigma_{\Delta_i}.$$

Let us prove that for non-degenerate affine toric varieties the component group of the automorphism group has the same description. To do this, let us show that
\begin{equation}
\label{eqrem}
    \mathrm{Aut}(\mathrm{Cl}(X))\cap \Sigma_D\simeq \mathrm{Aut}(N,\sigma) / \prod_{i=1}^s \Sigma_{\Delta_i}.
\end{equation}
Consider the group homomorphism $$\kappa: \mathrm{Aut}(N,\sigma)\to \mathrm{Aut}(\mathrm{Cl}(X)).$$ Each automorphism $\lambda\in \mathrm{Aut}(N,\sigma)$ corresponds to a $T$-equivariant automorphism ${\phi_{\lambda}\in\mathrm{Aut}(X)}$ according to~\cite[Theorem 3.3.4]{CLS}. By definition put $\kappa(\lambda)=\widetilde{\alpha}(\phi_\lambda)$. We see that $$\kappa(\mathrm{Aut}(N,\sigma))\subseteq \widetilde{\alpha}(\mathrm{Aut}(X))=\mathrm{Aut}(\mathrm{Cl}(X))\cap \Sigma_D.$$ In fact, an equality holds, because for any automorphism~$\phi\in\mathrm{Aut}(\mathrm{Cl}(X))$ such that $\phi$ permutes the elements $[D_1],\dots,[D_r]$ according to some permutation $\tau$, there exists a non-degenerate linear operator $L$ of the lattice $N$ that permutes the rays of the cone $\sigma$ according to $\tau^{-1}$, as shown in the proof of Theorem~\ref{crit}. Consequently, $\kappa(L^{-1})=\phi$, and the homomorphism $\kappa$ is surjective.

The kernel of $\kappa$ consists of the automorphisms of the lattice $N$ that permute the rays of the cone $\sigma$ corresponding to the equivalent prime divisors in the class group. Therefore, $\mathrm{Ker}\,\kappa=\prod_{i=1}^s \Sigma_{\Delta_i}.$ By the Fundamental theorem on homomorphisms, we have~(\ref{eqrem}).
\end{remark}

As a corollary of Theorem~\ref{thmcompgroup}, we provide another condition equivalent to the connectedness of the automorphism group.

\begin{sled}
    \label{sledcompgroup}
    Let $X$ be a non-degenerate affine toric variety. Then the group $\mathrm{Aut}(X)$ is connected if and only if there are no non-trivial automorphisms of the group $\mathrm{Cl}(X)$ that permute the elements $[D_1],\dots,[D_r]$.
\end{sled}
\section{Affine toric surfaces} \label{sectsurf}

The aim of this section is to apply the obtained results to study the connectedness of the automorphism group of a non-degenerate affine toric surface. Recall that the automorphism groups of such surfaces were described in~\cite[Theorem 4.2]{IVZaidenberg}.

Let $X$ be a non-degenerate affine toric surface. For an appropriate choice of a basis for~$N$ we may assume that the rational polyhedral cone $\sigma^{\vee}$ corresponding to the surface has the form $\langle (1,0),(a,b)\rangle$, where $b > a \geq 0$, and $(a,b)$ is a primitive vector, see~\cite[pp.~32--33]{Fulton}. Then the dual cone $\sigma$ is generated by the vectors
$v_1=(0,1)$ and $v_2=(b,-a)$:\\

\begin{center}
\begin{tikzpicture}[x=0.75pt,y=0.75pt,yscale=-1,xscale=1]
\draw  [draw opacity=0] (200.82,0.82) -- (461.05,0.82) -- (461.05,201.5) -- (200.82,201.5) -- cycle ; \draw  [color={rgb, 255:red, 155; green, 155; blue, 155 }  ,draw opacity=0.5 ] (200.82,0.82) -- (200.82,201.5)(220.82,0.82) -- (220.82,201.5)(240.82,0.82) -- (240.82,201.5)(260.82,0.82) -- (260.82,201.5)(280.82,0.82) -- (280.82,201.5)(300.82,0.82) -- (300.82,201.5)(320.82,0.82) -- (320.82,201.5)(340.82,0.82) -- (340.82,201.5)(360.82,0.82) -- (360.82,201.5)(380.82,0.82) -- (380.82,201.5)(400.82,0.82) -- (400.82,201.5)(420.82,0.82) -- (420.82,201.5)(440.82,0.82) -- (440.82,201.5)(460.82,0.82) -- (460.82,201.5) ; \draw  [color={rgb, 255:red, 155; green, 155; blue, 155 }  ,draw opacity=0.5 ] (200.82,0.82) -- (461.05,0.82)(200.82,20.82) -- (461.05,20.82)(200.82,40.82) -- (461.05,40.82)(200.82,60.82) -- (461.05,60.82)(200.82,80.82) -- (461.05,80.82)(200.82,100.82) -- (461.05,100.82)(200.82,120.82) -- (461.05,120.82)(200.82,140.82) -- (461.05,140.82)(200.82,160.82) -- (461.05,160.82)(200.82,180.82) -- (461.05,180.82)(200.82,200.82) -- (461.05,200.82) ; \draw  [color={rgb, 255:red, 155; green, 155; blue, 155 }  ,draw opacity=0.5 ]  ; 
\draw    (220.82,140.82) -- (300.82,20.82) ;
\draw    (220.82,140.82) -- (300.82,140.82) ;
\draw    (220.82,140.82) -- (259.71,82.48) ;
\draw [shift={(260.82,80.82)}, rotate = 123.69] [color={rgb, 255:red, 0; green, 0; blue, 0 }  ][line width=0.75]    (6.56,-1.97) .. controls (4.17,-0.84) and (1.99,-0.18) .. (0,0) .. controls (1.99,0.18) and (4.17,0.84) .. (6.56,1.97)   ;
\draw    (220.82,140.82) -- (238.82,140.82) ;
\draw [shift={(240.82,140.82)}, rotate = 180] [color={rgb, 255:red, 0; green, 0; blue, 0 }  ][line width=0.75]    (6.56,-1.97) .. controls (4.17,-0.84) and (1.99,-0.18) .. (0,0) .. controls (1.99,0.18) and (4.17,0.84) .. (6.56,1.97)   ; 
\draw    (380.82,140.82) -- (380.82,20.82) ;
\draw    (380.82,140.82) -- (439.15,179.71) ;
\draw [shift={(440.82,180.82)}, rotate = 213.69] [color={rgb, 255:red, 0; green, 0; blue, 0 }  ][line width=0.75]    (6.56,-1.97) .. controls (4.17,-0.84) and (1.99,-0.18) .. (0,0) .. controls (1.99,0.18) and (4.17,0.84) .. (6.56,1.97)   ; 
\draw    (380.82,140.82) -- (380.82,122.82) ;
\draw [shift={(380.82,120.82)}, rotate = 90] [color={rgb, 255:red, 0; green, 0; blue, 0 }  ][line width=0.75]    (6.56,-1.97) .. controls (4.17,-0.84) and (1.99,-0.18) .. (0,0) .. controls (1.99,0.18) and (4.17,0.84) .. (6.56,1.97)   ;

\draw (422.82,84.22) node [anchor=north west][inner sep=0.75pt]  [font=\footnotesize]  {$\sigma $};
\draw (282.82,84.22) node [anchor=north west][inner sep=0.75pt]  [font=\footnotesize]  {$\sigma ^{\vee }$};
\draw (222.82,64.22) node [anchor=north west][inner sep=0.75pt]  [font=\footnotesize]  {$M_{\mathbb{Q}}$};
\draw (342.82,64.22) node [anchor=north west][inner sep=0.75pt]  [font=\footnotesize]  {$N_{\mathbb{Q}}$};
\draw (210.82,104.22) node [anchor=north west][inner sep=0.75pt]  [font=\tiny]  {$( a,b)$};
\draw (222.82,144.22) node [anchor=north west][inner sep=0.75pt]  [font=\tiny]  {$( 1,0)$};
\draw (352.82,124.22) node [anchor=north] [inner sep=0.75pt]  [font=\tiny]  {$v_{1} =( 0,1)$};
\draw (362.82,164.22) node [anchor=north west][inner sep=0.75pt]  [font=\tiny]  {$v_{2} =( b,-a)$};
\draw (330.82,210.22) node [anchor=north] [inner sep=0.75pt]    {Figure 2.};
\end{tikzpicture}
\end{center}

Let us find the group $\mathrm{Cl}(X)$. Let the $T$-invariant prime divisors $D_1$ and $D_2$ correspond to the vectors $v_1$ and $v_2$, respectively. Then
$$\mathrm{Cl}(X)=\langle [D_1], [D_2] \rangle=\langle D_1,D_2 \rangle/ \langle \mathrm{div}(\chi^{(1,0)}),\mathrm{div}(\chi^{(0,1)}) \rangle.$$
Moreover, $$\mathrm{div}(\chi^{(1,0)})=\langle v_1,(1,0)\rangle D_1+ \langle v_2,(1,0)\rangle D_2=bD_2,$$
$$\mathrm{div}(\chi^{(0,1)})=\langle v_1,(0,1)\rangle D_1+ \langle v_2,(0,1)\rangle D_2=D_1-aD_2.$$
Consequently, $$\mathrm{Cl}(X)\simeq \mathbb{Z}/b\mathbb{Z}, \ [D_1]=a\in \mathbb{Z}/b\mathbb{Z},\ [D_2]=1\in \mathbb{Z}/b\mathbb{Z}.$$

\begin{propos}
\label{critsurf}
Let $X$ be a non-degenerate affine toric surface corresponding to the cone~$\sigma$ defined above. Then the group $\mathrm{Aut}(X)$ is connected if and only if one of the following three conditions $$(1)\ a=1;\ (2)\ b=1;\ (3)\ a^2\not\equiv 1\ (\mathrm{mod}\ b)$$ holds.
Moreover, if $\mathrm{Aut}(X)$ is not connected, then
$$\mathrm{Aut}(X)/\mathrm{Aut}(X)^0\simeq \mathbb{Z}/2\mathbb{Z}.$$
\end{propos}
\begin{proof}
Let us prove this statement using Corollary~\ref{sledcompgroup}. If the group $\mathrm{Aut}(X)$ is not connected, then there exists a non-trivial automorphism $\phi$ of the group $\mathrm{Cl}(X)\simeq \mathbb{Z}/b\mathbb{Z}$ that permutes elements $[D_1]=a\in\mathbb{Z}/b\mathbb{Z}$ and $[D_2]=1\in\mathbb{Z}/b\mathbb{Z}$. Consequently, $b\neq 1$, otherwise the class group is trivial, and $a\neq 1$, otherwise the classes of $D_1$ and $D_2$ coincide. Then the automorphism~$\phi$ acts as follows:
\begin{equation}
    \label{eqcritsurf0}
    \phi:\ 1\mapsto a,\ a\mapsto 1.
\end{equation}
Any automorphism of the group $\mathbb{Z}/b\mathbb{Z}$ acts as a multiplication by some invertible element. Therefore, we have the system
\begin{equation}
    \label{eqcritsurf}
    \begin{cases}
    a^2\equiv1\ (\mathrm{mod}\ b),\\
    a\neq 1,\\
    b\neq 1.
\end{cases}
\end{equation}
Conversely, if conditions~(\ref{eqcritsurf}) are satisfied, then there is an automorphism of $\mathrm{Cl}(X)$ that permutes $[D_1]$ and $[D_2]$, it is defined by the rule~(\ref{eqcritsurf0}).

In the case when the automorphism group of  $X$ is not connected, we have $$1< |\mathrm{Aut}(X)/\mathrm{Aut}(X)^0|\leq 2!$$ by Theorem~\ref{thmcompgroup}. Hence, the component group consists of two elements and is isomorphic to~$\mathbb{Z}/2\mathbb{Z}$. Therefore, Proposition~\ref{critsurf} is proved.
\end{proof}
\begin{remark}
We also provide a proof of the first part of Proposition~\ref{critsurf} using Theorem~\ref{crit}. This approach may be useful in case of varieties of higher dimension.

Let us prove the first part of Proposition~\ref{critsurf} using condition (2) of Theorem~\ref{crit}. If the automorphism group of $X$ is not connected, then there exists an automorphism of the Cox ring that normalizes, but does not preserve the $\mathrm{Cl}(X)$-grading. Denote this automorphism by $\phi$. Then the image of $\phi$ under $\gamma$ is a non-trivial automorphism of the class group. Note that it follows that $b \neq 1$, otherwise the class group is trivial and does not have non-trivial automorphisms. Any automorphism of the group $\mathbb{Z}/b\mathbb{Z}$ acts as a multiplication by some invertible element in $\mathbb{Z}/b\mathbb{Z}$. Let $\gamma(\phi)$ be multiplication by a invertible element $c \in \mathbb{Z}/b\mathbb{Z}$. This implies that $c$ and $b$ are coprime and $c \neq 1$.

From~\cite{Cox95}, we have $$\mathrm{deg}(T_1) = [D_1] = a \in \mathbb{Z}/b\mathbb{Z},\ \mathrm{deg}(T_2) = [D_2] = 1 \in \mathbb{Z}/b\mathbb{Z}.$$ Consequently, $\mathrm{deg}(\phi(T_1))$ coincides with $ac$ (mod $b$) in the group $\mathbb{Z}/b\mathbb{Z}$, and $\mathrm{deg}(\phi(T_2))$ coincides with $c$. Moreover, the Jacobian of the automorphism $\phi$ is a non-zero element of the field $\mathbb{K}$, so each of the elements $\phi(T_1)$ and $\phi(T_2)$ contains a linear term in $T_1$ or $T_2$. However, if $\phi(T_i)$ contains a linear term in $T_i$, then $c$ equals to 1, this gives a contradiction. Thus,
$$\phi(T_1) = k_1T_2 + \dots, \quad \phi(T_2) = k_2T_1 + \dots,$$
for some non-zero elements $k_1, k_2$ from the field $\mathbb{K}$. Therefore,
$$\mathrm{deg}(\phi(T_1)) = c  \mathrm{deg}(T_1) = \mathrm{deg}(T_2), \quad \mathrm{deg}(\phi(T_2)) = c  \mathrm{deg}(T_2) = \mathrm{deg}(T_1).$$
Consequently, we have
$$ac \equiv 1 , (\mathrm{mod} , b), \quad c = a, \quad c \neq 1.$$
Thus, the non-connectedness of the automorphism group of $X$ implies~(\ref{eqcritsurf}).

Let us prove the converse. Suppose conditions~(\ref{eqcritsurf}) hold. Consider the automorphism $\phi \in \widetilde{\mathrm{Aut}}(R(X))$ defined by: $$\phi(T_1) = T_2, \ \phi(T_2) = T_1.$$ The automorphism $\phi$ normalizes the $\mathrm{Cl}(X)$-grading, as it permutes the homogeneous components of $R(X)$ according to multiplication by $a$ in the divisor class group. At the same time, $\phi$ does not preserve the $\mathrm{Cl}(X)$-grading, since $a \neq 1$. Therefore, ${\widetilde{\mathrm{Aut}}(R(X)) \neq \mathrm{Ker}\,\gamma}$, and condition (2) of Theorem~\ref{crit} implies that the automorphism group of $X$ is not connected.

Let us provide a proof of the first part of Proposition~\ref{critsurf} using condition (3) of Theorem~\ref{crit}. Suppose that the automorphism group of $X$ is not connected. Then there exists a linear operator $L \in \mathrm{GL}_2(\mathbb{Z})$, $L(\sigma) = \sigma$ such that ${L(v_1) = v_2}$, ${L(v_2) = v_1}$, but $[D_1] \neq [D_2]$. Note that it follows that $b \neq 1$, otherwise the class group is trivial and does not have distinct elements. We have $a \neq 1$, otherwise $[D_1] = [D_2]$. Furthermore, we observe that $$(1,0) = \frac{v_2 + av_1}{b}.$$ Thus, from the linearity of $L$, we have $$L((1,0)) = \frac{L(v_2) + aL(v_1)}{b} = \frac{(ab, 1-a^2)}{b} = \left(a, \frac{1-a^2}{b}\right).$$ Since $L \in \mathrm{GL}_2(\mathbb{Z})$, we have $\frac{1-a^2}{b} \in \mathbb{Z}$ and $a^2 \equiv 1\  (\mathrm{mod} \ b)$. Consequently, the non-connectedness of the automorphism group of $X$ implies the system~(\ref{eqcritsurf}).

Conversely, suppose that conditions~(\ref{eqcritsurf}) are satisfied. Then to prove that $\mathrm{Aut}(X)$ is not connected, it suffices to find an operator $L \in \mathrm{GL}_2(\mathbb{Z})$ satisfying condition (3) of Theorem~\ref{crit}. A required map is the operator $\widetilde{L}$ defined on the basis vectors as follows: $$\widetilde{L}: (1,0) \mapsto \left(a, \frac{1-a^2}{b}\right), \  (0,1) \mapsto (b,-a).$$ Indeed, $\widetilde{L}$ permutes the vectors $v_1$ and $v_2$, and $[D_1] \neq [D_2]$.

\end{remark}

\section{Examples} \label{sectex}
In this section we provide some examples illustrating the obtained results.

\begin{example}
\label{ex0}
The automorphism group of the affine space $\mathbb{A}^n$ is connected by~\cite[Theorem 6]{Popov}. This fact also follows from Theorem~\ref{crit}, since the affine space is a factorial toric variety, and the divisor class group of a factorial variety is trivial, see, for example,~\cite[Theorem 4.0.18 (b)]{CLS}.
\end{example}
 
\begin{example}
\label{ex1}
Consider the variety $$X_1=\mathbb{V}(xy-z^2)\subset \mathrm{Spec}(\mathbb{K}[x,y,z]).$$
The variety $X_1$ is a non-degenerate affine toric surface, corresponding to the rational polyhedral cone $\sigma_1^{\vee}$ generated by the vectors $(1,0)$ and $(1,2)$, see Figure 3. In notation of Section~\ref{sectsurf}, we have $a=1$ and $b=2$. Therefore, by Proposition~\ref{critsurf}, the automorphism group of $X_1$ is connected: $$\mathrm{Aut}(X_1)=\mathrm{Aut}(X_1)^0.$$
\end{example}

\begin{example}
\label{ex2}
Let $$X_2=\mathbb{V}(xy-z^3)\subset \mathrm{Spec}(\mathbb{K}[x,y,z]).$$
This variety is also a non-degenerate affine toric surface. It corresponds to the cone~$\sigma_2^{\vee}$ generated by the vectors $(1,0)$ and $(2,3)$, see Figure 3. By Proposition~\ref{critsurf}, the automorphism group of $X_2$ is not connected and contains exactly two connected components: $$\mathrm{Aut}(X_2)/\mathrm{Aut}(X_2)^0 \simeq \mathbb{Z}/2\mathbb{Z}.$$ In this case the $T$-invariant prime divisors are $$D_1=\{y=z=0\}\text{ and }D_2=\{x=z=0\}.$$ Note that the automorphism $$x\mapsto y,\ y \mapsto x,\ z\mapsto z$$ is not contained in the neutral component by Proposition~\ref{thmAut0Kera}.
\end{example}

\begin{example}
\label{ex3}
Consider $$X_3=\mathbb{V}(xy-z^2, wz-y^3)\subset \mathrm{Spec}(\mathbb{K}[x,y,z, w]).$$ Note that $X_3$ is a non-degenerate affine toric surface, corresponding to the cone~$\sigma_3^{\vee}$ in Figure~3, generated by the vectors $(1,0)$ and $(2,5)$. By Proposition \ref{critsurf}, the automorphism group of~$X_3$ is connected: $$\mathrm{Aut}(X_3)=\mathrm{Aut}(X_3)^0.$$
\end{example}

\tikzset{every picture/.style={line width=0.75pt}} 
\begin{center}

\begin{tikzpicture}[x=0.75pt,y=0.75pt,yscale=-1,xscale=1]

\draw  [draw opacity=0] (149.82,0.82) -- (510.11,0.82) -- (510.11,161.51) -- (149.82,161.51) -- cycle ; \draw  [color={rgb, 255:red, 155; green, 155; blue, 155 }  ,draw opacity=0.5 ] (149.82,0.82) -- (149.82,161.51)(169.82,0.82) -- (169.82,161.51)(189.82,0.82) -- (189.82,161.51)(209.82,0.82) -- (209.82,161.51)(229.82,0.82) -- (229.82,161.51)(249.82,0.82) -- (249.82,161.51)(269.82,0.82) -- (269.82,161.51)(289.82,0.82) -- (289.82,161.51)(309.82,0.82) -- (309.82,161.51)(329.82,0.82) -- (329.82,161.51)(349.82,0.82) -- (349.82,161.51)(369.82,0.82) -- (369.82,161.51)(389.82,0.82) -- (389.82,161.51)(409.82,0.82) -- (409.82,161.51)(429.82,0.82) -- (429.82,161.51)(449.82,0.82) -- (449.82,161.51)(469.82,0.82) -- (469.82,161.51)(489.82,0.82) -- (489.82,161.51)(509.82,0.82) -- (509.82,161.51) ; \draw  [color={rgb, 255:red, 155; green, 155; blue, 155 }  ,draw opacity=0.5 ] (149.82,0.82) -- (510.11,0.82)(149.82,20.82) -- (510.11,20.82)(149.82,40.82) -- (510.11,40.82)(149.82,60.82) -- (510.11,60.82)(149.82,80.82) -- (510.11,80.82)(149.82,100.82) -- (510.11,100.82)(149.82,120.82) -- (510.11,120.82)(149.82,140.82) -- (510.11,140.82)(149.82,160.82) -- (510.11,160.82) ; \draw  [color={rgb, 255:red, 155; green, 155; blue, 155 }  ,draw opacity=0.5 ]  ;
\draw    (169.82,140.82) -- (229.82,20.82) ;
\draw    (169.82,140.82) -- (249.82,140.82) ;
\draw    (169.82,140.82) -- (188.92,102.61) ;
\draw [shift={(189.82,100.82)}, rotate = 116.57] [color={rgb, 255:red, 0; green, 0; blue, 0 }  ][line width=0.75]    (6.56,-1.97) .. controls (4.17,-0.84) and (1.99,-0.18) .. (0,0) .. controls (1.99,0.18) and (4.17,0.84) .. (6.56,1.97)   ;
\draw    (169.82,140.82) -- (187.82,140.82) ;
\draw [shift={(189.82,140.82)}, rotate = 180] [color={rgb, 255:red, 0; green, 0; blue, 0 }  ][line width=0.75]    (6.56,-1.97) .. controls (4.17,-0.84) and (1.99,-0.18) .. (0,0) .. controls (1.99,0.18) and (4.17,0.84) .. (6.56,1.97)   ;
\draw    (289.66,141.41) -- (369.98,20.23) ;
\draw    (289.82,140.82) -- (369.82,140.82) ;
\draw    (289.82,140.82) -- (329.04,81.31) ;
\draw [shift={(330.14,79.64)}, rotate = 123.38] [color={rgb, 255:red, 0; green, 0; blue, 0 }  ][line width=0.75]    (6.56,-1.97) .. controls (4.17,-0.84) and (1.99,-0.18) .. (0,0) .. controls (1.99,0.18) and (4.17,0.84) .. (6.56,1.97)   ;
\draw    (289.82,140.82) -- (307.82,140.82) ;
\draw [shift={(309.82,140.82)}, rotate = 180] [color={rgb, 255:red, 0; green, 0; blue, 0 }  ][line width=0.75]    (6.56,-1.97) .. controls (4.17,-0.84) and (1.99,-0.18) .. (0,0) .. controls (1.99,0.18) and (4.17,0.84) .. (6.56,1.97)   ;
\draw    (410,141) -- (458.11,20.4) ;
\draw    (410,141) -- (490,141) ;
\draw    (409.82,140.82) -- (449.08,42.68) ;
\draw [shift={(449.82,40.82)}, rotate = 111.8] [color={rgb, 255:red, 0; green, 0; blue, 0 }  ][line width=0.75]    (6.56,-1.97) .. controls (4.17,-0.84) and (1.99,-0.18) .. (0,0) .. controls (1.99,0.18) and (4.17,0.84) .. (6.56,1.97)   ;
\draw    (410,141) -- (428,141) ;
\draw [shift={(430,141)}, rotate = 180] [color={rgb, 255:red, 0; green, 0; blue, 0 }  ][line width=0.75]    (6.56,-1.97) .. controls (4.17,-0.84) and (1.99,-0.18) .. (0,0) .. controls (1.99,0.18) and (4.17,0.84) .. (6.56,1.97)   ;

\draw (231.82,64.22) node [anchor=north west][inner sep=0.75pt]  [font=\footnotesize]  {$\sigma _{1}^{\lor }$};
\draw (171.82,64.22) node [anchor=north west][inner sep=0.75pt]  [font=\footnotesize]  {$M_{\mathbb{Q}}$};
\draw (153.82,106.72) node [anchor=north west][inner sep=0.75pt]  [font=\tiny]  {$( 1,2)$};
\draw (171.82,144.22) node [anchor=north west][inner sep=0.75pt]  [font=\tiny]  {$( 1,0)$};
\draw (351.5,65.4) node [anchor=north west][inner sep=0.75pt]  [font=\footnotesize]  {$\sigma _{2}^{\lor }$};
\draw (291.5,65.4) node [anchor=north west][inner sep=0.75pt]  [font=\footnotesize]  {$M_{\mathbb{Q}}$};
\draw (280.82,104.22) node [anchor=north west][inner sep=0.75pt]  [font=\tiny]  {$( 2,3)$};
\draw (291.82,144.22) node [anchor=north west][inner sep=0.75pt]  [font=\tiny]  {$( 1,0)$};
\draw (471.68,65.58) node [anchor=north west][inner sep=0.75pt]  [font=\footnotesize]  {$\sigma _{3}^{\lor }$};
\draw (411.82,64.22) node [anchor=north west][inner sep=0.75pt]  [font=\footnotesize]  {$M_{\mathbb{Q}}$};
\draw (392.82,105.72) node [anchor=north west][inner sep=0.75pt]  [font=\tiny]  {$( 2,5)$};
\draw (411.84,144.99) node [anchor=north west][inner sep=0.75pt]  [font=\tiny]  {$( 1,0)$};
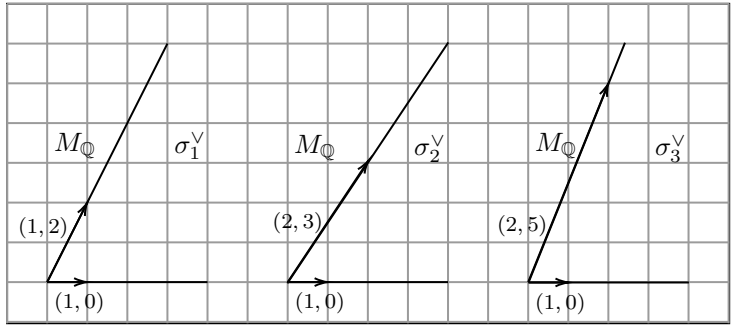
\draw (175.5,170.9) node [anchor=north west][inner sep=0.75pt]    {Figure 3. Cones $\sigma _{1}^{\lor } ,\ \sigma _{2}^{\lor }$ and $\sigma _{3}^{\lor }$ from examples~\ref{ex1},~\ref{ex2} and~\ref{ex3}};

\end{tikzpicture}
    
\end{center}

\begin{example}
\label{ex4}
Consider the non-degenerate non-simplicial affine toric variety $$X_4=\mathbb{V}(xy-zw)\subset \mathrm{Spec}(\mathbb{K}[x,y,z, w]).$$ It corresponds to the cone $\sigma_4^{\vee}$ in Figure 4.

The primitive vectors on the rays of the cone $\sigma_4$ dual to $\sigma_4^{\vee}$ are
$$v_1=
\begin{pmatrix} 1 \\ -1 \\ 0 \end{pmatrix},\ v_2=\begin{pmatrix} 1 \\ 0 \\ -1 \end{pmatrix},\ v_3=\begin{pmatrix} 0 \\ 1 \\ 0 \end{pmatrix},\ v_4=\begin{pmatrix} 0 \\ 0 \\ 1 \end{pmatrix}.$$ 
These vectors correspond to the $T=(\mathbb{K}^{\times})^3$-invariant prime divisors: $$D_1=\{ x=w=0\},\ D_2=\{x=z=0\},\ D_3=\{z=y=0\},\ D_4=\{y=w=0\}.$$
Let us find the divisor class group of the variety $X_4$.

The group of $T$-invariant principal divisors on $X_4$ is generated by the following Weil divisors: $$\mathrm{div}(\chi^{(0,0,1)})=\sum_{i=1}^{4}\langle v_i,(0,0,1) \rangle D_i=D_4-D_2,$$
$$\mathrm{div}(\chi^{(0,1,0)})=\sum_{i=1}^{4}\langle v_i,(0,1,0) \rangle D_i=D_3-D_1,$$
$$\mathrm{div}(\chi^{(1,0,0)})=\sum_{i=1}^{4}\langle v_i,(1,0,0) \rangle D_i=D_1+D_2.$$
Therefore,
$$\mathrm{Cl}(X_4)\simeq \langle D_1,D_2,D_3,D_4 \rangle /(D_4-D_2,D_3-D_1,D_1+D_2)\simeq \langle [D_1] \rangle \simeq \mathbb{Z},$$
where $$[D_1]=-[D_2]=[D_3]=-[D_4]=1\in \mathbb{Z}.$$ The group $\mathrm{Cl}(X_4)$ has a unique non-trivial automorphism $\phi$, which permutes the elements of the set $\{[D_1],[D_2],[D_3],[D_4]\}$. It acts as a multiplication by $-1$ in the group $\mathbb{Z}$. By Theorem~\ref{thmcompgroup}, we have $$\mathrm{Aut}(X_4)/\mathrm{Aut}(X_4)^0\simeq \mathbb{Z}/2\mathbb{Z}.$$

Note that by Proposition~\ref{thmAut0Kera} the neutral component of the automorphism group of $X_4$ contains the automorphism $$x\mapsto y,\ y\mapsto x,\ z\mapsto w,\ w\mapsto z$$ and does not contain the automorphism $$x\mapsto y,\ y\mapsto x,\ z\mapsto z,\ w\mapsto w.$$
\end{example}

\tikzset{every picture/.style={line width=0.75pt}} 

\begin{center}
\begin{tikzpicture}[x=0.75pt,y=0.75pt,yscale=-1,xscale=1]

\draw    (240.11,123.94) -- (378.11,123.94) ;
\draw [shift={(380.11,123.94)}, rotate = 180] [color={rgb, 255:red, 0; green, 0; blue, 0 }  ][line width=0.75]    (6.56,-1.97) .. controls (4.17,-0.84) and (1.99,-0.18) .. (0,0) .. controls (1.99,0.18) and (4.17,0.84) .. (6.56,1.97)   ;
\draw    (240.11,123.94) -- (318.7,45.36) ;
\draw [shift={(320.11,43.94)}, rotate = 135] [color={rgb, 255:red, 0; green, 0; blue, 0 }  ][line width=0.75]    (6.56,-1.97) .. controls (4.17,-0.84) and (1.99,-0.18) .. (0,0) .. controls (1.99,0.18) and (4.17,0.84) .. (6.56,1.97)   ;
\draw    (240.11,123.94) -- (378.48,25.11) ;
\draw [shift={(380.11,23.94)}, rotate = 144.46] [color={rgb, 255:red, 0; green, 0; blue, 0 }  ][line width=0.75]    (6.56,-1.97) .. controls (4.17,-0.84) and (1.99,-0.18) .. (0,0) .. controls (1.99,0.18) and (4.17,0.84) .. (6.56,1.97)   ;
\draw    (320.11,43.94) -- (380.11,23.94) ;
\draw    (320.11,43.94) -- (320.11,143.94) ;
\draw    (380.11,23.94) -- (380.11,123.94) ;
\draw    (320.11,143.94) -- (380.11,123.94) ;
\draw [color={rgb, 255:red, 128; green, 128; blue, 128 }  ,draw opacity=1 ]   (240.11,123.94) -- (400.11,163.94) ;
\draw [color={rgb, 255:red, 128; green, 128; blue, 128 }  ,draw opacity=1 ]   (240.11,123.94) -- (420.11,63.94) ;
\draw [color={rgb, 255:red, 128; green, 128; blue, 128 }  ,draw opacity=1 ]   (240.11,123.94) -- (240.11,3.94) ;
\draw    (240.11,123.94) -- (318.17,143.46) ;
\draw [shift={(320.11,143.94)}, rotate = 194.04] [color={rgb, 255:red, 0; green, 0; blue, 0 }  ][line width=0.75]    (6.56,-1.97) .. controls (4.17,-0.84) and (1.99,-0.18) .. (0,0) .. controls (1.99,0.18) and (4.17,0.84) .. (6.56,1.97)   ;

\draw (301,149.4) node [anchor=north west][inner sep=0.75pt]  [font=\tiny]  {$( 1,0,0)$};
\draw (355,54.4) node [anchor=north west][inner sep=0.75pt]  [font=\footnotesize]  {$\sigma _{4}^{\lor }$};
\draw (378.5,128.44) node [anchor=north west][inner sep=0.75pt]  [font=\tiny]  {$( 1,1,0)$};
\draw (369,10.44) node [anchor=north west][inner sep=0.75pt]  [font=\tiny]  {$( 1,1,1)$};
\draw (293.5,30.9) node [anchor=north west][inner sep=0.75pt]  [font=\tiny]  {$( 1,0,1)$};
\draw (214,170.9) node [anchor=north west][inner sep=0.75pt]    {Figure 4. Cone $\sigma _{4}^{\lor }$ from example~\ref{ex4}};

\end{tikzpicture}

\end{center}

\begin{example}
    Finally, let us provide an example of an affine toric variety with a non-commutative component group of the automorphism group. Let the cone $\sigma_5$ be generated in the three-dimensional rational vector space by the vectors $$v_1=
    \begin{pmatrix}
    2 \\ 0 \\ 1
\end{pmatrix}
    ,\ v_2=\begin{pmatrix} 0 \\ 2 \\ 1 \end{pmatrix} \text{ and } v_3=\begin{pmatrix} 0 \\ 0 \\ 1 \end{pmatrix}.$$ Consider the affine toric variety $X_5$ corresponding to the cone $\sigma_5$. Here the class group is isomorphic to $\mathbb{Z}/2\mathbb{Z} \oplus \mathbb{Z}/2\mathbb{Z}$, and $$[D_1]=(1,0),\ [D_2]=(0,1),\ [D_3]=(1,1)\in \mathbb{Z}/2\mathbb{Z} \oplus \mathbb{Z}/2\mathbb{Z}.$$ The automorphism group of $\mathrm{Cl}(X_5)$ is isomorphic to the group of permutations of a set of three elements, and each automorphism leaves the set $\{[D_1],[D_2], [D_3]\}$ invariant. Therefore, by Theorem~\ref{thmcompgroup}, $$\mathrm{Aut}(X_5)/\mathrm{Aut}(X_5)^0\simeq \mathrm{Aut}(\mathrm{Cl}(X_5))\cap \Sigma_D\simeq S_3.$$ Note that in this case, the upper bound on the number of connected components of the automorphism group from Theorem~\ref{thmcompgroup} is achieved.
\end{example}

\end{document}